\documentclass{article}
\usepackage{epsf}
\usepackage{latexsym}
\usepackage{amsfonts}
\usepackage{amsmath}

\newtheorem{theorem}{Theorem}[section]

\newenvironment{proof}{{\sc Proof:}}{~\hfill $\quad\Box$}

\begin{document}

\title{The Nature of Length, Area, and Volume in Taxicab Geometry}

\author{Kevin P. Thompson\thanks{North Arkansas College}}

\date{}

\thispagestyle{empty}
\renewcommand\thispagestyle[1]{} 

\maketitle
\begin{abstract}
While the concept of straight-line length is well understood in taxicab geometry, little research has been done into the length of curves or the nature of area and volume in this geometry. This paper sets forth a comprehensive view of the basic dimensional measures in taxicab geometry.
\end{abstract}

\section{Introduction}

When taxicab geometry is mentioned, the context is often a contrast with two-dimensional Euclidean geometry. Whatever the emphasis, the depth of the discussion frequently does not venture far beyond how the measurement of length is affected by the usual Euclidean metric being replaced with the taxicab metric where the distance between two points $(x_1,y_1)$ and $(x_2,y_2)$ is given by
\[
d_t=|x_2-x_1|+|y_2-y_1|
\]

Papers have appeared that explore length in taxicab geometry in greater detail \cite{Ozcan,OzcanKaya}, and a few papers have derived area and volume formulae for certain figures and solids using exclusively taxicab measurements \cite{Colakoglu,Kaya,OzcanKayaArea}. However, this well-built foundation has left a number of open questions. What is the taxicab length of a curve in two dimensions? In three dimensions? What does area actually mean in two-dimensional taxicab geometry? Is the taxicab area of a ``flat'' surface in three dimensions comparable to the taxicab area of the ``same'' surface (in a Euclidean sense) in two dimensions? These are all fundamental questions that do not appear to have been answered by the current body of research in taxicab geometry. In this paper we wish to provide a comprehensive, unified view of length, area, and volume in taxicab geometry.

Where suitable and enlightening, we will use the value $\pi_t=4$ as the value for $\pi$ in taxicab geometry \cite{Euler}.

\section{The Nature of Length}

The first dimensional measure we will examine is the simplest of the measures: the measurement of length. This is also, for line segments at least, the most well understood measure in taxicab geometry.

\subsection{One-dimensional Length}

The simplest measurement of length is the length of a line segment in one dimension. On this point, Euclidean and taxicab geometry are in complete agreement. The length of a line segment from a point $A=a$ to another point $B=b$ is simply the number of unit lengths covered by the line segment,
\[
d_e(A,B)=d_t(A,B)=|b-a|
\]

\subsection{Two-dimensional Linear Length}

In two dimensions, however, the Euclidean and taxicab metrics are not always in agreement on the length of a line segment. For line segments parallel to a coordinate axis, such as the line segment with endpoints $A=(x_1,y)$ and $B=(x_2,y)$, there is agreement since both metrics reduce to one-dimensional measurement: $d_e(A,B)=d_t(A,B)=|x_2-x_1|$. Only when the line segment is not parallel to one of the coordinate axes do we finally see disagreement between the Euclidean and taxicab metrics. The taxicab length of such a line segment can be viewed as the sum of the Euclidean lengths of the projections of the line segment onto the coordinate axes (Figure \ref{taxicab_length}),
\begin{equation}
\label{eq_taxicab_length}
d_t=d_e|\cos\theta|+d_e|\sin\theta|
\end{equation}

\begin{figure}[b]
\centerline{\epsffile{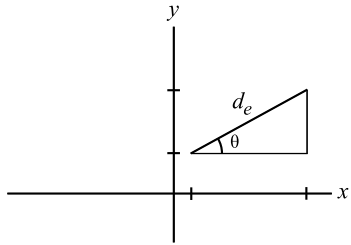}}
\caption{Taxicab length as a sum of Euclidean projections onto the coordinate axes.}
\label{taxicab_length}
\end{figure}

\noindent The Pythagorean Theorem tells us the Euclidean and taxicab lengths will generally not agree for line segments that are not parallel to one of the coordinate axes. Line segments of the same Euclidean length will have various taxicab lengths as their position relative to the coordinate axes changes. If one were to place a scale on a diagonal line, the Euclidean and taxicab markings would differ with the largest discrepancy being at a $45^{\circ}$ angle to the coordinate axes (Figure \ref{scale_diff}).

\begin{figure}[t]
\centerline{\epsffile{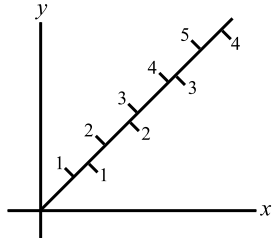}}
\caption{The scale of measurement differs along a diagonal line for Euclidean geometry (below the line) and taxicab geometry (above the line).}
\label{scale_diff}
\end{figure}

These cases and types of length measurement are well known and are well understood to those familiar with taxicab geometry. But, even in two dimensions, there is at least one other type of length measurement in Euclidean geometry: the length of a curve. How is the length of a (functional) curve in two dimensions measured in taxicab geometry?

\subsection{Arc Length in Two Dimensions}

In Euclidean geometry, the arc length of a curve described by a function $f$ with a continuous derivative over an interval $[a,b]$ is given by
\[
L=\int_a^b \! \sqrt{1+[f'(x)]^2} \, dx
\]

\begin{figure}[b]
\centerline{\epsffile{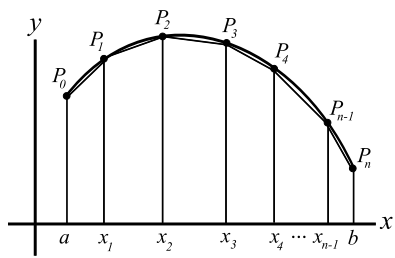}}
\caption{Approximating arc length in taxicab geometry.}
\label{arc_length_2d}
\end{figure}

\noindent To find a corresponding formula in taxicab geometry, we may follow a traditional method for deriving arc length. First, split the interval $[a,b]$ into $n$ subintervals with widths $\triangle x_1,\triangle x_2,\ldots,\triangle x_n$ by defining points $x_1,x_2,...,x_{n-1}$ between $a$ and $b$ (Figure \ref{arc_length_2d}). This allows for definition of points $P_0,P_1,P_2,...,P_n$ on the curve whose $x$-coordinates are $a,x_1,x_2,...,x_{n-1},b$. By connecting these points, we obtain a polygonal path approximating the curve. At this crucial point of the derivation we apply the taxicab metric instead of the Euclidean metric to obtain the taxicab length $L_k$ of the $k$th line segment,
\[
L_k=\triangle x_k+|f(x_k)-f(x_{k-1})|
\]
By the Mean Value Theorem, there is a point $x_k^*$ between $x_{k-1}$ and $x_k$ such that $f(x_k)-f(x_{k-1})=f'(x_k^*)\triangle x_k$. Thus, the taxicab length of the $k$th segment becomes
\[
L_k=\triangle x_k+|f'(x_k^*)\triangle x_k|=(1+|f'(x_k^*)|)\, \triangle x_k
\]
and the taxicab length of the entire polygonal path is
\[
\sum_{k=1}^n L_k=\sum_{k=1}^n (1+|f'(x_k^*)|)\, \triangle x_k
\]
By increasing the number of subintervals while forcing $\max \triangle x_k \to 0$, the length of the polygonal path will approach the arc length of the curve. So,
\[
L=\lim_{\max \triangle x \to 0} \sum_{k=1}^n (1+|f'(x_k^*)|)\, \triangle x_k
\]
The right side of this equation is a definite integral, so the taxicab arc length of the curve described by the function $f$ over the interval $[a,b]$ is given by 
\begin{equation}
\label{eq_arc_length_2d}
L=\int_a^b \! \left(1+|f'(x)|\right) \, dx
\end{equation}

As an example, the northeast quadrant of a taxicab circle of radius $r$ centered at the origin is described by $f\left(x\right)=-x+r$. The arc length of this curve over the interval $[0,r]$ is given by
\[
L=\int_{0}^{r} \! \left(1+\left|-1\right|\right) \, dx=2r
\]
which is precisely one-fourth of the circumference of a taxicab circle of radius $r$. A more interesting application involves the northeast quadrant of the Euclidean circle of radius $r$ at the origin described by $f\left(x\right)=\sqrt{r^{2}-x^{2}}$.
\begin{eqnarray*}
L & = & \int_{0}^{r} \! \left(1+\left|-x\left(r^{2}-x^{2}\right)^{-\frac{1}{2}}\right|\right) \, dx\\
 & = & \left.\left(x-\sqrt{r^{2}-x^{2}}\right) \, \right|_{0}^{r}\\
 & = & r-0-0+r\\
 & = & 2r
\end{eqnarray*}

\begin{figure}[t]
\centerline{\epsffile{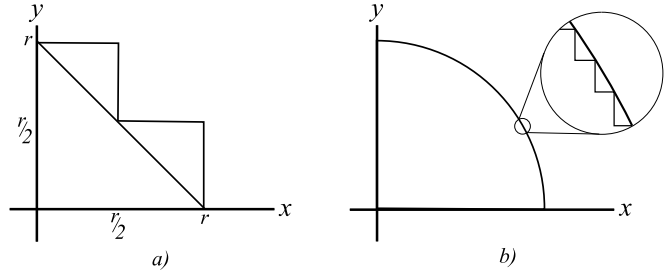}}
\caption{a) Multiple straight-line paths between two points can have the same length in taxicab geometry. b) A curve can be approximated by small horizontal and vertical paths.}
\label{different_paths}
\end{figure}

The distance is the same as if we had traveled along the taxicab circle between the two endpoints! While this is a shocking result to the Euclidean observer who is accustomed to distinct paths between two points generally having different lengths, the taxicab observer merely shrugs his shoulders. A curve such as a Euclidean circle can be approximated with increasingly small horizontal and vertical steps near its path (Figure \ref{different_paths}b). As any good introduction to taxicab geometry teaches us, such as Krause \cite{Krause}, multiple straight-line paths between two points have the same length in taxicab geometry (see Figure \ref{different_paths}a). So, each of these approximations of the Euclidean circle will have the same taxicab length. Therefore, we should expect the limiting case to also have the same length. To further make the point, we can follow the Euclidean parabola $f(x)=-\frac{1}{r}x^{2}+r$ in the first quadrant and arrive at the same result (Figure \ref{multiple_curves}).
\[
L=\int_{0}^{r} \! \left(1+\left|\frac{-2x}{r}\right|\right) \, dx=r+r=2r
\]

\begin{figure}[h]
\centerline{\epsffile{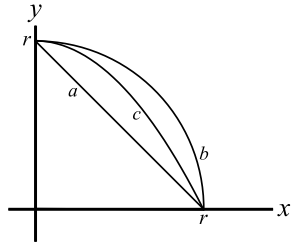}}
\caption{Multiple curves between two points can have the same length in taxicab geometry. Shown are a) part of the taxicab circle of radius $r$, b) part of the Euclidean circle of radius $r$, and c) part of the curve $-\frac{1}{r}x^{2}+r$.}
\label{multiple_curves}
\end{figure}

To formalize these observations, we have the following theorem.

\begin{theorem}
\label{thm_multiple_paths}
If a function $f$ is monotone increasing or decreasing and differentiable with a continuous derivative over an interval $[a,b]$, then the arc length of $f$ over $[a,b]$ is
\[
L=\left(b-a\right)+\left|f\left(b\right)-f\left(a\right)\right|
\]
(i.e. the path from $(a,v=f(a))$ to $(b,w=f(b))$ is independent of the function $f$ under the stated conditions).
\end{theorem}
\begin{proof}
The taxicab arc length of $f$ over $[a,b]$ is
\[
L=\int_a^b \! \left(1+\left|f'\left(x\right)\right|\right) \, dx=\left(b-a\right)+\int_a^b \! \left(\left|f'\left(x\right)\right|\right) \, dx
\]
If f is monotone increasing, $f'\left(x\right)\geq0$. And, since f has a continuous derivative, we may apply the First Fundamental Theorem of Calculus to get
\[
L=\left(b-a\right)+f\left(b\right)-f\left(a\right)
\]
If f is monotone decreasing, we get the similar result
\[
L=\left(b-a\right)+f\left(a\right)-f\left(b\right)
\]
In either case, the latter difference must be positive so we have 
\[
L=\left(b-a\right)+\left|f\left(b\right)-f\left(a\right)\right|
\]
\end{proof}

For functions that are not monotone increasing or decreasing, the arc length can be found using this theorem by taking the sum of the arc lengths over the subintervals where the function is monotone increasing or decreasing.

\subsection{Three-Dimensional Linear Length}

The taxicab length of a line segment in three dimensions is a natural extension of the formula in two dimensions. For a line segment with endpoints $(x_1,y_1,z_1)$ and $(x_2,y_2,z_2)$,
\[
d_t=|x_2-x_1|+|y_2-y_1|+|z_2-z_1|
\]
A nice discussion of the three-dimensional metric and other properties is given in \cite{Akca}. There are no surprises here, and the taxicab length can be viewed as the sum of the Euclidean lengths of the projections of the line segment onto the three coordinate axes.

\subsection{Arc Length in Three Dimensions}

The simplest extension of the arc length of a curve to three dimensions is for parametric curves. For a curve in two dimensions, we can parameterize the function $f=(x(t),y(t))$ and modify Equation (\ref{eq_arc_length_2d}) to get
\[
L=\int_a^b \! \left(\left|\frac{dx}{dt}\right|+\left|\frac{dy}{dt}\right|\right) \, dt
\]
Generalizing this to three dimensions, we have
\[
L=\int_a^b \! \left(\left|\frac{dx}{dt}\right|+\left|\frac{dy}{dt}\right|+\left|\frac{dz}{dt}\right|\right) \, dt
\]
Intuitively, Theorem \ref{thm_multiple_paths} generalizes to three-dimensions saying that the length of a three-dimensional ``monotonic'' curve depends only on its endpoints.

\section{The Nature of Area}

In prior research regarding area in taxicab geometry \cite{Kaya,OzcanKayaArea}, the underlying assumption has been that the area of a figure should agree in Euclidean and taxicab geometry, although the formulae to compute the area could be vastly different. Additionally, it appears area has only been investigated in two dimensions. In this section we wish to examine this assumption about area and explore surface area in three dimensions.

Figure \ref{fig_square_area} gets right to the question of what area means in taxicab geometry. The figures shown are each squares in both geometries: all sides and angles have the same measure. If we cling to the notion that the area of a square is the square of the length of its sides, these figures have equal Euclidean area (2) but very different taxicab areas (4 and 2). The position of the sides of the square relative to the coordinate axes in taxicab geometry affects their length which in turn affects the area calculation.

\begin{figure}[h]
\centerline{\epsffile{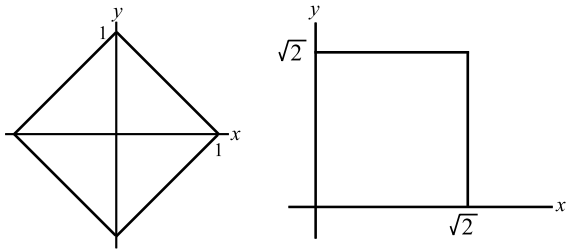}}
\caption{Squares with equal Euclidean area (2) but different taxicab area (4 and 2, respectively) when viewing area as the square of the length of the sides.}
\label{fig_square_area}
\end{figure}

We have a choice before us. We can cling to the area of a square being the square of its sides and live with the position of a figure affecting its taxicab area just as the position of a line segment affects its taxicab length. Or, we can maintain consistency with the Euclidean area of the figure and, as in this example, seek a new formula for the area of a square. Several points here and later in this paper will hopefully help us to make a firm decision here.

Reviewing the nature of length in taxicab geometry, we see agreement with Euclidean geometry on the length of a line segment (a one-dimensional ``figure'') in one dimension. Only for line segments in two dimensions, a dimension higher than the figure, do we see discrepancies between the geometries.

If we apply this same logic to area, the geometries should agree on area (a two-dimensional figure) in two dimensions. Only for the area of a surface in three dimensions would we expect to see differences because the figure would have a different position relative to the coordinate planes.

In addition, the area of a figure has traditionally been viewed as the number of square units of the plane enclosed by the figure. Since Euclidean and taxicab geometry are built on the Cartesian coordinate system, it seems logical the computation of area in two dimensions should agree in these two geometries.

With this reasoning, we proceed with the traditional Euclidean concept of area in two dimensions while looking for the proper extension of the concept to three dimensions. Therefore, for our example, we would need a new formula for the taxicab area of a square. (For the interested reader, it can be shown that the area of a square in terms of the taxicab length $s$ of its sides is $s^2(\cos_t^2\theta+\sin_t^2\theta)$ where $\theta$ is the taxicab angle between one of the sides and the $x$-axis and the trigonometric functions are taxicab \cite{ThompsonDray}.)

\subsection{Area in Two Dimensions}

The standard approach to the computation of two-dimensional area in undergraduate calculus courses is integration. Estimates of area under a functional curve involve increasing numbers of rectangles that each approximate a portion of the area under the curve. Since Euclidean and taxicab geometry agree on length in the horizontal and vertical directions, these area approximations in the two geometries will agree. Therefore, we should expect the calculation of area under a curve by integration to transfer seamlessly to taxicab geometry. This is a further argument to view area in two dimensions consistently between the two geometries.

\subsection{Surface Area in Three Dimensions}

As stated in a previous section, we should expect the taxicab area of a ``flat'' surface in three dimensions to differ from the area of the ``same'' figure (in a Euclidean sense) in one of the coordinate planes. This follows from the pattern seen with line segments in one and two dimensions. To discover how the area will differ, consider a plane rotated upward from the $xy$-plane along the $y$-axis (Figure \ref{fig_rotated_plane}).

\begin{figure}[h]
\centerline{\epsffile{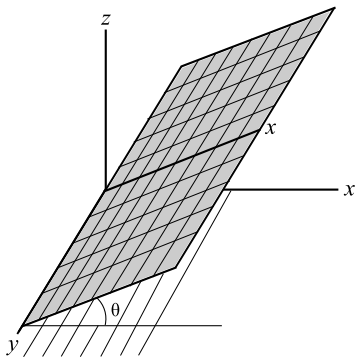}}
\caption{A plane rotated up from the $xy$-plane along the $y$-axis will have different Euclidean and taxicab scales of measurement in the $x$-direction.}
\label{fig_rotated_plane}
\end{figure}

In a manner similar to Figure \ref{scale_diff}, the taxicab scale in the $x$-direction will be compressed in the rotated plane compared to the Euclidean scale.  The angle $\theta$ formed in a plane parallel to the $xz$-plane will determine the scaling effect seen in the one dimension of the rotated plane with the maximal difference occurring at $45^{\circ}$. The result is the ``same'' figure (in a Euclidean sense) in the $xy$-plane and in the rotated plane will have different taxicab areas. (Interestingly, the taxicab coordinate grid in the rotated plane is not even composed of Euclidean squares but of Euclidean rectangles since the taxicab scale in the $y$-direction is unaffected by the specific rotation performed.)

As a specific example, a unit square at the origin in the $xy$-plane will have a Euclidean and taxicab area of 1. If we rotate the square upward $45^{\circ}$ along the $y$-axis, the Euclidean area will be unchanged but the taxicab area will now be $\sqrt{2}$ since one side will now have taxicab length $\cos 45^{\circ}+\sin 45^{\circ}=\sqrt{2}$.

If the plane is instead rotated along the $x$-axis, the taxicab area is again found by scaling the Euclidean area. When rotation occurs in both manners, the scaling factors compound each other. So, the taxicab area of a figure in the rotated plane can be found from its Euclidean area using Equation (\ref{eq_taxicab_length}) as a guide for each direction of rotation: the scaling factors are sums of the (Euclidean) cosine and sine of the angles $\alpha$ and $\beta$ formed by lines in cross-sectional planes parallel to the $xz$-plane and the $yz$-plane similar to the single rotation illustrated in Figure \ref{fig_rotated_plane},
\begin{equation}
\label{eq_scaling}
A_t=A_e(|\cos_{e}\alpha|+|\sin_{e}\alpha|)(|\cos_{e}\beta|+|\sin_{e}\beta|)
\end{equation}

\subsection{Surfaces of Revolution}

\begin{figure}[b]
\centerline{\epsffile{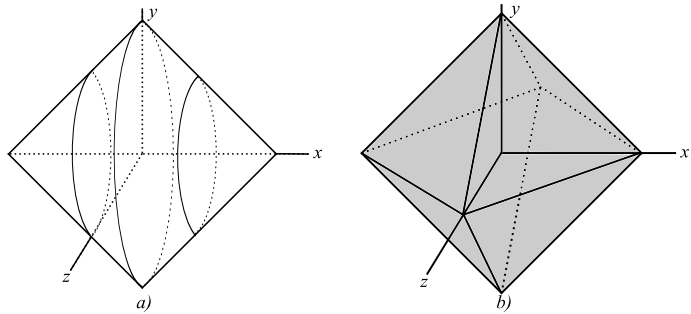}}
\caption{a) A Janssen taxicab ``sphere'' and b) a taxicab sphere.}
\label{fig_spheres}
\end{figure}

In \cite{Janssen} Janssen makes several attempts to revolve half a taxicab circle about an axis to obtain a taxicab sphere. The first approach revolves the upper half of a taxicab circle in a \emph{Euclidean} manner about the horizontal axis resulting in a ``sphere'' with some straight edges but also with some curved surfaces (Figure \ref{fig_spheres}a). This does not seem to quite fit the ``squarish'' nature of seemingly everything in taxicab geometry and does not appear to satisfy the usual definition of a sphere (all points equidistant from the center).

The author's line of thought is quite instructive, however. If we consider a Euclidean solid of revolution about the $x$-axis, the cross-section of the solid in the $yz$-plane is a Euclidean circle. Therefore, for a taxicab revolution about the $x$-axis, the cross-section should be a {\em taxicab} circle in that plane. A taxicab sphere obtained by revolving a taxicab circle about the horizontal axis in a taxicab manner would yield Figure \ref{fig_spheres}b. This completely ``squarish'' object seems to fit the geometry better and indeed does, by construction, satisfy the usual definition of a sphere. So, revolution about the horizontal axis should follow the path of a taxicab circle in the $yz$-plane and not the path of a Euclidean circle.

As an application of our approach to surface area in three dimensions, we would like to find a formula for the surface area of a solid of revolution. For a differentiable function $f$ over an interval $[a,b]$ divide the interval into $n$ subintervals. Using a linear approximation of the curve over each subinterval, we revolve each line segment about the $x$-axis. In the Euclidean derivation, this creates the frustrum of a cone. In taxicab geometry, the discussion above shows us this will create the frustrum of a taxicab cone, which is equivalent to the frustrum an Euclidean right square pyramid (partially shown in Figure \ref{fig_revolve}). To find the taxicab surface area of the frustrum, we will compute the Euclidean area and then scale the area by examining the rotations of the frustrum sides to the $xy$-plane.

\begin{figure}[b]
\centerline{\epsffile{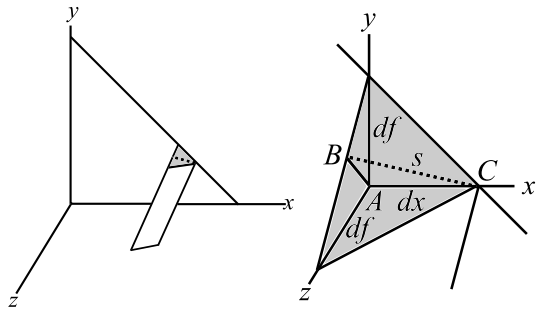}}
\caption{A portion of the revolution of a line segment about the $x$-axis in a taxicab manner with a close-up of the shaded triangle. (The full revolution creates the frustrum of an Euclidean right square pyramid.)}
\label{fig_revolve}
\end{figure}

The Euclidean surface area of the frustrum of a pyramid is $\frac{1}{2}(p_1+p_2)s$ where $p_1$ and $p_2$ are the perimeters of the ``top'' and ``bottom'' edges of the frustrum and $s$ is the slant height. For the $k$th subinterval, the Euclidean perimeters are $4\sqrt{2}f(x_{k-1})$ and $4\sqrt{2}f(x_k)$. To compute the slant height, we project the shaded triangle in Figure \ref{fig_revolve} onto a plane parallel to the $yz$-plane. The slant height $s$ is the hypotenuse of the right triangle $ABC$. The projected triangle is an isoceles right triangle with legs of length $|f(x_{k})-f(x_{k-1})|$ (labeled $df$ in the figure), so the altitude of this triangle is $\frac{\sqrt{2}}{2}|f(x_{k})-f(x_{k-1})|$. The other leg of triangle $ABC$ has length $\triangle x_k$ (labeled $dx$ in the figure). Therefore, the slant height is
\[
s=\sqrt{(\triangle x_k)^2+\frac{1}{2}[f(x_{k})-f(x_{k-1})]^2}
\]
So, we have the Euclidean surface area of $k$th frustrum of the approximation:
\begin{eqnarray*}
S_{k_e} & = & \frac{1}{2}(p_1+p_2)s\\
        & = & \frac{1}{2}(4\sqrt{2}f(x_{k-1})+4\sqrt{2}f(x_k))\sqrt{(\triangle x_k)^2+\frac{1}{2}[f(x_{k})-f(x_{k-1})]^2}\\
        & = & 2\sqrt{2}(f(x_{k-1})+f(x_k))\sqrt{(\triangle x_k)^2+\frac{1}{2}[f(x_{k})-f(x_{k-1})]^2}
\end{eqnarray*}

The sides of this frustrum form Euclidean angles of $45^{\circ}$ with the $xy$-plane using a cross-sectional plane parallel to the $xz$-plane. This gives one scaling factor for the taxicab surface area of $\cos 45^{\circ}+\sin 45^{\circ}=\sqrt{2}$. The other scaling factor is dependent on the sum of the cosine and sine of the angle of the linear approximation of the function and the $x$-axis. This gives another scaling factor of
\[
\frac{\triangle x_k+|f(x_k)-f(x_{k-1})|}{\sqrt{(\triangle x_k)^2+[f(x_k)-f(x_{k-1})]^2}}
\]
Therefore, using Equation (\ref{eq_scaling}) the taxicab surface area of the $k$th frustrum of the approximation is
\[
{\textstyle
S_k=\frac{4(f(x_{k-1})+f(x_k))(\triangle x_k+|f(x_k)-f(x_{k-1})|)\sqrt{(\triangle x_k)^2+\frac{1}{2}[f(x_{k})-f(x_{k-1})]^2}}{\sqrt{(\triangle x_k)^2+[f(x_k)-f(x_{k-1})]^2}}
}
\]
As the derivation now continues along the same lines as the Euclidean version, the Intermediate Value Theorem and the Mean Value Theorem give us values $x_k^*$ and $x_k^{**}$ such that $f(x_k^*)=\frac{1}{2}(f(x_k)+f(x_{k-1}))$ and $f(x_k)-f(x_{k-1})=f'(x_k^{**})\triangle x_k$. Therefore,
\begin{eqnarray*}
S_k & = & \frac{8f(x_k^*)(\triangle x_k+|f'(x_k^{**})|\triangle x_k)\sqrt{(\triangle x_k)^2+\frac{1}{2}[f'(x_k^{**})\triangle x_k]^2}}{\sqrt{(\triangle x_k)^2+[f'(x_k^{**})\triangle x_k]^2}}\\
    & = & \frac{8f(x_k^*)(1+|f'(x_k^{**})|)\sqrt{1+\frac{1}{2}[f'(x_k^{**})]^2}}{\sqrt{1+[f'(x_k^{**})]^2}} \triangle x_k\\
    & = & 8f(x_k^*)(1+|f'(x_k^{**})|)\sqrt{1-\frac{[f'(x_k^{**})]^2}{2(1+[f'(x_k^{**})]^2)}} \triangle x_k
\end{eqnarray*}
Accumulating the frustrums and taking the limit yields the formula for the taxicab surface area of a solid of revolution:
\begin{equation}
\label{eq_surface_area}
S = \int_{a}^{b} \! 2\pi_t f(x)(1+|f'(x)|)\sqrt{1-\frac{[f'(x)]^2}{2(1+[f'(x)]^2)}} \, dx
\end{equation}

This is a very interesting formula. The portion outside the radical is very reminiscent of the Euclidean formula with the expected changes in the circle circumference factor ($2\pi_t f(x)$) and the arc length factor ($1+|f'(x)|$) due to the taxicab metric. The extra radical represents a scaling factor based on the ratio of the slant height to the linear approximation of the curve. The greater the difference between these two quantities the more the frustrum sides are rotated with respect to the $xy$-plane thus requiring more scaling. If these two quantities are close, the frustrum sides are rotated very little and therefore require little scaling.

To complete the analysis inspired by Janssen, half of a taxicab sphere of radius $r$ is obtained by revolving the function $f(x)=-x+r$ over the interval $[0,r]$. Doubling this result yields the surface area of a whole taxicab sphere:
\begin{eqnarray*}
S & = & 2\int_{0}^{r} \! 2\pi_t (-x+r)(1+|-1|)\sqrt{1-\frac{(-1)^2}{2(1+(-1)^2)}} \, dx\\
 & = & 4\pi_t \sqrt{3}\int_{0}^{r} \! (-x+r) \, dx\\
 & = & -2\pi_t \sqrt{3}\left.(-x+r)^2\right|_{0}^{r}\\
 & = & 2\pi_t \sqrt{3}r^2
\end{eqnarray*}
Such a taxicab sphere is composed of two Euclidean right square pyramids of height $\frac{r\sqrt{6}}{2}$ and base side length $r\sqrt{2}$. The Euclidean surface area is therefore $4r^2\sqrt{3}$. Since the sides are at $45^{\circ}$ angles to the $xy$-plane in both directions, the taxicab area scaling factors are both $\sqrt{2}$. This gives a taxicab surface area of $8r^2\sqrt{3}$ in perfect agreement with the formula obtained from the solid of revolution.

\section{The Nature of Volume}

If we follow the same pattern used for length and area measures, we expect the Euclidean and taxicab volume of a solid (a three-dimensional object) in three dimensions to be equal. Only in four dimensions would we expect orientation of the solid to impact the taxicab volume when compared to the Euclidean volume.

\subsection{Solids of Revolution}

To compute the volume of a Euclidean solid of revolution, the traditional ``slicing'' method involves computing the volume of an infinitessimal cylinder as $dx$ multipled by the area of the Euclidean circle of revolution. This method can be adapted directly into taxicab geometry by replacing the Euclidean circle area with the taxicab circle area \emph{if} the volume of a taxicab cylinder is found in the same manner.

A taxicab cylinder resembles an Euclidean cylinder except that the cross section is a taxicab circle (Figure \ref{fig_cylinder}). For such a cylinder lying along the $x$-axis, the cross-sectional area will lie in a plane parallel to the $yz$-plane and therefore the Euclidean and taxicab areas will agree. In addition, the height of the cylinder will agree in the two geometries as it lies along a coordinate axis. Based on our assumption that volume will agree in three dimensions between the geometries, the taxicab volume of a cylinder is the product of the area of the cross-sectional taxicab circle and the height.

\begin{figure}[t]
\centerline{\epsffile{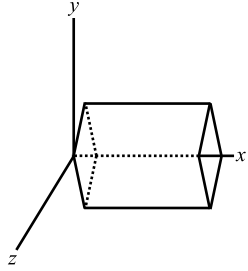}}
\caption{A taxicab cylinder.}
\label{fig_cylinder}
\end{figure}

A taxicab circle of radius $r$ is composed of four isoceles triangles with area $\frac{1}{2}r^2$ (Proposition 1 of \cite{OzcanKayaArea}). Therefore, a taxicab cylinder of radius $r$ and height $h$ will have volume
\[
V_t=2r^2h=\frac{1}{2}\pi_t r^2h
\]

With this information we can now compute the volume of a solid of revolution obtained by revolving a continuous function $f$ about the $x$-axis over an interval $[a,b]$.
\begin{equation}
\label{eq_solid_volume}
V_t=\int_{a}^{b} \! \frac{1}{2}\pi_t [f(x)]^2 \, dx
\end{equation}

Continuing our primary example, the upper half of a taxicab circle of radius $r$ centered at the origin is described by
\[
f(x)=
\begin{cases}
x+r & -r\leq x \leq 0 \\
-x+r & 0 < x \leq r
\end{cases}
\]
Using Equation (\ref{eq_solid_volume}) the volume of a taxicab sphere obtained by revolving the upper half of the circle about the $x$-axis is
\begin{eqnarray*}
V & = & \int_{-r}^{r} \! \frac{1}{2}\pi_t [f(x)]^{2} \, dx\\
 & = & \frac{1}{2}\pi_t \int_{-r}^{0} \! (x+r)^{2} \, dx + \frac{1}{2}\pi_t \int_{0}^{r} \! (-x+r)^{2} \, dx\\
 & = & \left.\frac{1}{6}\pi_t (x+r)^{3} \, \right|_{-r}^{0} - \left.\frac{1}{6}\pi_t (-x+r)^{3} \, \right|_{0}^{r}\\
 & = & \frac{1}{3}\pi_t r^{3}
\end{eqnarray*}
This result agrees precisely with the volume of a taxicab sphere as a special case of a tetrahedron \cite{Colakoglu} providing some assurance of our concept of volume and computational approach.

It should also be noted that the surface area of a sphere is not the derivative of the volume with respect to the radius. This is a consequence of the radius of the sphere not being everywhere perpendicular to the surface.

\section{Other Examples}

To conclude our discussion, we can derive surface area and volume formulae for the taxicab equivalents of a few common solids.

\begin{figure}[b]
\centerline{\epsffile{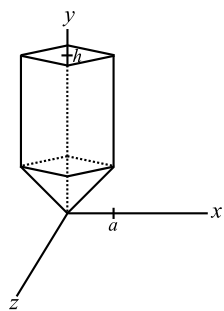}}
\caption{A taxicab paraboloid.}
\label{fig_paraboloid}
\end{figure}

Using the standard definition of a taxicab parabola described in \cite{Laatsch}, half of a (horizontally) parallel case of the parabola with focus $(0,a)$ and directrix $y=-a$ is given by
\[
f(y)=\begin{cases}
y, & 0\leq y\leq a\\
a, & y>a\end{cases}
\]
Restricting the total ``height'' of the parabola to $h$ and revolving the curve about the $y$-axis yields an open-top taxicab paraboloid (Figure \ref{fig_paraboloid}) surface area
\begin{eqnarray*}
S & = & \int_{0}^{a} \! 2\pi_t y(2)\sqrt{1-\frac{(1)^2}{2(1+(1)^2)}} \, dy + \int_{a}^{h} \! 2\pi_t a(1)\sqrt{1-\frac{(0)^2}{2(1+(0)^2)}} \, dy\\
  & = & 2\pi_t\sqrt{3}\int_{0}^{a} \! y \, dy+2\pi_t a\int_{a}^{h} \! dy\\
  & = & \left.\pi_t\sqrt{3} y^{2} \, \right|_{0}^{a}+\left.2\pi_t ay \, \right|_{a}^{h}\\
  & = & \pi_t a^2\sqrt{3}+2\pi_t a(h-a)
\end{eqnarray*}
As we would expect based on Figure \ref{fig_paraboloid}, this is the sum of the surface area of a cylinder with radius $a$ and height $h-a$ and half a sphere of radius $a$.

For the volume of a paraboloid we have
\begin{eqnarray*}
V & = & \int_{0}^{h} \! \frac{1}{2}\pi_t \left(f\left(y\right)\right)^{2} \, dy\\
  & = & \int_{0}^{a} \! \frac{1}{2}\pi_t y^{2} \, dy+\int_{a}^{h} \! \frac{1}{2}\pi_t a^{2} \, dy\\
  & = & \left.\frac{1}{6}\pi_t y^{3} \, \right|_{0}^{a}+\left.\frac{1}{2}\pi_t a^{2}y \, \right|_{a}^{h}\\
  & = & \frac{1}{6}\pi_t a^{3}+\frac{1}{2}\pi_t a^{2}(h-a)
\end{eqnarray*}
Again, as expected, this is the sum of the volume of a cylinder of radius $a$ and height $h-a$ and half a sphere of radius $a$.

In the defining paper concerning conics in taxicab geometry \cite{KayaConics}, nondegenerate (or ``true'') two-foci taxicab ellipses are described as taxicab circles, hexagons, and octagons. If we revolve half of one of these figures about the $x$-axis we obtain a taxicab ellipsoid (Figure \ref{fig_ellipsoid}). If we consider a taxicab ellipse with major axis length $2a$, minor axis length $2b$, and $s$ the sum of the distances from a point on the ellipse to the foci, the function
\[
f(x)=
\begin{cases}
x+\frac{s}{2} & -a\leq x < b-\frac{s}{2} \\
b & b-\frac{s}{2} \leq x < -b+\frac{s}{2} \\
-x+\frac{s}{2} & -b+\frac{s}{2} \leq x \leq a
\end{cases}
\]
describes the upper half of an ellipse centered at the origin. This function will generally cover the sphere ($s=2a$ and $a=b$), hexagon ($s=2a$ and $a>b$), and octagon ($s>2a$) cases for a taxicab ellipse. The ellipsoid solid of revolution will have volume
\begin{eqnarray*}
V & = & \frac{1}{2}\pi_t \int_{-a}^{b-\frac{s}{2}} \! \left(x+\frac{s}{2}\right)^2 \, dx + \frac{1}{2}\pi_t \int_{b-\frac{s}{2}}^{-b+\frac{s}{2}} \! b^2 \, dx + \frac{1}{2}\pi_t \int_{-b+\frac{s}{2}}^{a} \! \left(-x+\frac{s}{2}\right)^2 \, dx\\
  & = & \frac{1}{3}\pi_t b^3-\frac{1}{6}(s-2a)^3+\frac{1}{2}\pi_t b^2(s-2b)
\end{eqnarray*}
For the case of a spherical taxicab ellipsoid (Figure \ref{fig_ellipsoid}c), this formula reduces to $\frac{1}{3}\pi_t b^3$ in agreement with our previous result. For the case of a hexagon (Figure \ref{fig_ellipsoid}b), this formula reduces to $\frac{1}{3}\pi_t b^3+\frac{1}{2}\pi_t b^2(2a-2b)$ which is the sum of the volume of a taxicab sphere of radius $b$ and a taxicab cylinder of radius $b$ and height $2a-2b$. This is to be expected since a hexagonal taxicab ellipsoid is composed of a taxicab cylinder capped by two taxicab half-spheres.

\begin{figure}[t]
\centerline{\epsffile{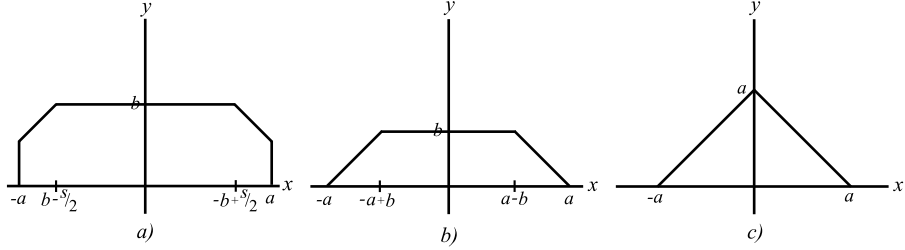}}
\caption{Nondegenerate, two-foci taxicab ellipses. Shown are the upper halves of a) the octogonal ellipse, b) the hexagonal ellipse, and c) the circular ellipse (a taxicab circle).}
\label{fig_ellipsoid}
\end{figure}

In Euclidean geometry, there is not a closed form for the surface area an ellipsoid. Since a taxicab ellipse is composed of straight lines, this problem is avoided in taxicab geometry. The surface area of an taxicab ellipsoidal solid of revolution is
\begin{eqnarray*}
S & = & 2\pi_t \int_{-a}^{b-\frac{s}{2}} \! \left(x+\frac{s}{2}\right)(1+|1|)\sqrt{1-\frac{(1)^2}{2(1+(1)^2)}} \, dx\\
  & & {} + 2\pi_t \int_{b-\frac{s}{2}}^{-b+\frac{s}{2}} \! b(1+0)\sqrt{1-\frac{(0)^2}{2(1+(0)^2)}} \, dx\\
  & & {} + 2\pi_t \int_{-b+\frac{s}{2}}^{a} \! \left(-x+\frac{s}{2}\right)(1+|-1|)\sqrt{1-\frac{(-1)^2}{2(1+(-1)^2)}} \, dx\\
  & & {} + 2(\frac{1}{2}\pi_t (\frac{s}{2}-a)^2)\\
  & = & 2\pi_t \sqrt{3}b^2-2\pi_t \sqrt{3}(\frac{s}{2}-a)^2+2\pi_t b(s-2b)+\pi_t (\frac{s}{2}-a)^2
\end{eqnarray*}
For a hexagonal ellipsoid, the first term is the surface area of the ends which combine to equal a taxicab sphere of radius $b$; the second and last terms are zero; and, the third term is the area of a taxicab cylinder of radius $b$ and length $s-2b$ composing the middle of the ellipsoid. For the octogonal ellipsoid (Figure \ref{fig_ellipsoid}a), the first term is an overestimate since the top of the sphere is not present on either end. This is corrected by the subtraction of the second term amounting to a taxicab sphere of radius $\frac{s}{2}-a$. (A similar correcion term is seen in the volume formula above.) The last term accounts for the area of the ends of the ellipsoid which are taxicab circles of radius $\frac{s}{2}-a$.

\section{Conclusion}

Generalizing the patterns we have set forth in this paper, it appears that the taxicab measure of a $n$-dimensional figure in $n$-dimensional space will agree with the Euclidean measure of the figure. However, in $(n+1)$-dimensional space, the taxicab measure of the figure will in general depend on its position in the space when compared with the Euclidean measure. Along the way we generalized the observation that multiple paths between two points can have the same taxicab length and described a strategy for dealing with figures living in a dimension higher than themselves.  We also developed in taxicab geometry the strategy of revolving a function about the $x$-axis axis by observing that the shape of the cross-section of the solid in the $yz$-plane should be a taxicab circle. As in Euclidean geometry, this has yielded a simple method for deriving surface area and volume formulas for some standard taxicab solids created as solids of revolution.

\medskip

{\em \noindent The author would like to thank Tevian Dray of Oregon State University for his help and suggestions during the writing and revision of this paper.}

\end{document}